\documentclass{article}
\usepackage[utf8]{inputenc}
\usepackage{tikz}
\usepackage{fullpage}
\usepackage{amsthm,amsmath,amssymb}
\usepackage{xcolor}
\usepackage{enumerate}
\usepackage[colorlinks=true,linkcolor=teal,citecolor=teal,filecolor=teal,urlcolor=teal]{hyperref}
\usepackage{cleveref}

\newtheorem*{theorem*}{Theorem}
\newtheorem{theorem}{Theorem}[section]
\newtheorem{conj}[theorem]{Conjecture}
\newtheorem{lemma}[theorem]{Lemma}

\newtheorem{problem}[theorem]{Problem}
\newtheorem{cor}[theorem]{Corollary}
\newtheorem{prop}[theorem]{Proposition}
\newtheorem{proposition}[theorem]{Proposition}

\newcommand{\CC}{\mathcal{C}}
\newcommand{\PP}{\mathcal{P}}
\newcommand{\FF}{\mathcal{F}}
\newcommand{\RR}{\mathcal{R}}
\newcommand{\HH}{\mathcal{H}}
\newcommand{\GG}{\mathcal{G}}
\newcommand{\lpt}{\mathrm{lpt}}
\newcommand{\lct}{\mathrm{lct}}
\newcommand{\comment}[1]{}
\newcommand{\set}[1]{\left\{#1\right\}}

\renewcommand{\ss}{\subseteq}
\renewcommand{\emptyset}{\varnothing}
\renewcommand{\setminus}{-}
\newcommand{\ceil}[1]{\left \lceil #1 \right \rceil}
\newcommand{\floor}[1]{\left \lfloor #1 \right \rfloor}
\newcommand{\bag}[1]{B(#1)}

\newcommand{\subtree}[1]{S(#1)}
\newcommand{\st}{\colon\,}

\newcommand{\leafage}[1]{\ell(#1)}
\newcommand{\tw}[1]{\mathrm{tw}(#1)}

\newcommand{\mmf}[1]{f(#1)}
\newcommand{\ccg}[1]{\mathrm{ccg}(#1)}

\newcommand{\descendants}[3]{D_{#3}(#1,#2)}

\title{Longest Path and Cycle Transversals in Chordal Graphs}
\author{James A. Long Jr. \thanks{Department of Computer Science \& Mathematics, Fairmont State University, Fairmont, WV, USA. Email: james.long@fairmontstate.edu} \and Kevin G. Milans \thanks{School of Mathematical and Data Sciences, West Virginia University, Morgantown, WV, USA. Email: milans@math.wvu.edu} \and Michael C. Wigal \thanks{Department of Mathematics, University of Illinois at Urbana-Champaign, Urbana, IL, USA. Email: wigal@illinois.edu}}
\date{\today}

\begin{document}

\maketitle

\begin{abstract}
    We show that if $G$ is a $n$-vertex connected chordal graph, then it admits a longest path transversal of size $O(\log^2 n)$. Under the stronger assumption of 2-connectivity, we show $G$ admits a longest cycle transversal of size $O(\log n)$. We also provide longest path and longest cycle transversals which are bounded by the leafage of the chordal graph.
\end{abstract}

\section{Introduction}

Gallai \cite{G68} asked whether the intersection of all longest paths in a connected graph is nonempty.  Motivated by this question, a \emph{Gallai vertex} in a graph is a vertex belonging to every longest path. The question was resolved in the negative by Walther~\cite{W69}, who constructed a connected graph with no Gallai vertices.  Smaller examples of connected graphs with no Gallai vertices were found by both Walther and Voss~\cite{WV74} and 
Zamfirescu~\cite{Z76}. One such counterexample is the Petersen fragment, which is obtained from splitting an arbitrary vertex of the Petersen graph into three degree one vertices, see \Cref{fig:test}.  There has been much work understanding how longest paths intersect, with some surveys being available~\cite{SZZ13,Z01}.

A family of graphs $\GG$ is \emph{Gallai} if each connected $G\in\GG$ has at least one Gallai vertex.  There is an ongoing line of research to determine which graph families are Gallai, see for example \cite{BGLS04,CFCGGL20,CL20,CEFHSYY17,JKLW16,J15,KP90}.  Given a set of graphs $\FF$, a graph $G$ is \emph{$\FF$-free} of it does not contain any member of $\FF$ as an induced subgraph.  There has been some work toward a characterization of the families $\FF$ for which the $\FF$-free graphs are Gallai when $\FF$ is very small, see \cite{GS18,LN24,LMM23}, with the full characterization remaining an open question.

Given a family of sets $\FF$, the \emph{intersection graph} on $\FF$ is the graph with vertex set $\FF$ with vertices $u$ and $v$ adjacent if and only if $u$ and $v$ have non-empty intersection.  A graph $G$ is an \emph{interval graph} if $G$ is the intersection graph of a family of closed intervals in $\mathbb{R}$.  A \emph{chord} of a cycle $C$ in a graph $G$ is an edge $e \in E(G) - E(C)$ with both endpoints on $C$.  A graph $G$ is \emph{chordal} if each cycle on at least $4$ vertices has a chord.  In 2004, Balister, Gy\H{o}ri, Lehel, and Schelp~\cite{BGLS04} showed that the interval graphs form a Gallai family, and suggested that the same might be true for the larger family of chordal graphs. The question remains open.

There is some partial evidence that the family of chordal graphs are Gallai. A graph $G$ is a \emph{split graph} if $V(G)$ can be partitioned into a clique and an independent set. All split graphs are chordal.  Bender,  Richmond, and Wormald \cite{BRW85} proved that a uniformly sampled labeled chordal graph almost surely splits. As Klav\v{z}ar and Petkov\v{s}ek~\cite{KP90} showed split graphs are Gallai, we may conclude almost all (labeled) chordal graphs are also Gallai.

A \emph{longest path transversal} in $G$ is a set of vertices that intersects every longest path in $G$.  The \emph{longest path transversal number} of $G$, denoted $\lpt(G)$, is the minimum size of a longest path transversal.  Showing that $\GG$ is a Gallai family amounts to proving $\lpt(G)=1$ for each connected $G\in\GG$.  Kierstead and Ren~\cite{KR2023} proved that $\lpt(G)\le 5n^{2/3}$ when $G$ is a connected $n$-vertex graph, improving on an earlier bound of $\lpt(G) \le 8n^{3/4}$ due to Long, Milans, and Munaro~\cite{LMM21}.  From below, Grünbaum~\cite{GRUNBAUM} constructed a connected graph $G$ with $\lpt(G)=3$.  There is no known connected graph $G$ with $\lpt(G)\ge 4$.  We state our first main result, where $\lg x$ denotes $\log_2 x$. 

\begin{theorem*}[See \Cref{thm:lpt_bound}]
    If $G$ is a connected $n$-vertex chordal graph, then $\lpt(G)\le 4\lg^2 n + O(\log n)$.
\end{theorem*}

A \emph{tree} is a connected acyclic graph. In 1974, Gavril~\cite{GAVRIL74} showed that a graph $G$ is chordal if and only if $G$ is the intersection graph of subtrees of a tree $T$. In this case, we say $T$ is a \emph{host tree} for the chordal graph $G$. The subtrees of a tree are well known to have the \emph{Helly property} (e.g., see \cite{G04}), meaning that if $\FF$ is a set of subtrees of $T$ and every pair of trees in $\FF$ has nonempty intersection, then some vertex in $T$ is common to each subtree in $\FF$.  Since the longest paths in a connected graph are pairwise intersecting, showing that $\GG$ is Gallai is equivalent to showing that for each connected $G\in\GG$, the longest paths in $G$ have the Helly property.  As the longest paths in a connected graph are pairwise intersecting, the longest paths in a tree are pairwise intersecting subtrees, and it follows from the Helly property that the trees form a Gallai family.

The treewidth of a graph is a fundamental graph parameter introduced in the graph minors project of Robertson and Seymour \cite{RS86}. In nonrigorous terms, treewidth is a measure of how far a graph is from being a tree. Treewidth has received significant attention due to its applications in fixed parameter tractable algorithms; many problems which are NP-hard in general have polynomial algorithms when restricted to graphs with bounded treewidth.  The \emph{treewidth} of a graph $G$, denoted $\tw{G}$, can be defined as $\tw{G} = \min\{ \omega(H) - 1 : G \subseteq H \text{ and } H \text{ is chordal}\}$ where $\omega(H)$ denotes the size of the largest clique in $H$.  In particular for a chordal graph $G$, we have that $\tw{G} = \omega(G) - 1$. For a reference on treewidth, see \cite{R97}.

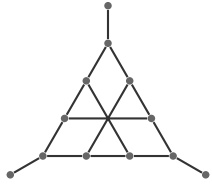
\begin{figure}
    \begin{center}
    \begin{tikzpicture}
        \tikzset{vertex/.style={circle,fill=black!60,inner sep=1pt,minimum size=0.5ex}}
        \tikzset{edge/.style={draw,line width=0.8pt,black!80}}

		\begin{scope}[yshift=-0.2cm]
		\begin{scope}[every node/.style={vertex},rotate=90]
			\foreach \n in {0,...,2}
			{{
				\pgfmathsetmacro{\ang}{\n * 360/3}
				\draw (\ang:1.5cm) node[vertex] (U\n) {} ;
			}}
			
			\pgfmathsetmacro{\increment}{sqrt(3)/3}
			\path ++(0:1cm) node (V0) {} 
				++(150:\increment cm) node (V1) {} 
				++(150:\increment cm) node (V2) {}
				++(150:\increment cm) node (V3) {} 
				++(-90:\increment cm) node (V4) {} 
				++(-90:\increment cm) node (V5) {} 
				++(-90:\increment cm) node (V6) {} 
				++(30:\increment cm) node (V7) {} 
				++(30:\increment cm) node (V8) {} ;
				
			\foreach \m in {0,...,8}
			{{
				\pgfmathsetmacro{\suc}{int(mod(\m + 1, 9))}
				\draw[edge] (V\m) -- (V\suc) ;
			}}
			
		\end{scope}
		
		\draw[edge] (U0) edge (V0) (U1) edge (V3) (U2) edge (V6) ;
		\draw[edge] (V1) -- (V5) (V2) -- (V7) (V4) -- (V8) ;
	\end{scope}
	\end{tikzpicture}
 \end{center}
 \caption{Petersen fragment}\label{fig:test}
\end{figure}

A \emph{bramble} of a graph is a set of connected subgraphs which pairwise intersect or are joined by an edge.  The \emph{order} of a bramble $B$ is the minimum size of a set of vertices intersecting each subgraph in $B$.   It is well known~\cite{ST93} that $\tw{G}+1$ is equal to the maximum order of a bramble in $G$. As observed by Rautenbach and Sereni \cite{RS14}, as the family of longest paths in a connected graph form a bramble, it follows that $\lpt(G) \le \tw{G}+1$ when $G$ is connected.  If $G$ is connected and chordal, then $\lpt(G) \le \tw{G} + 1 = \omega(G)$. Later, Harvey and Payne~\cite{HP23} improved this to $\lpt(G) \le 4\ceil{(\omega(G))/5}$ when $G$ is a connected chordal graph.  

A closely related problem is finding a small set of vertices intersecting every longest cycle, called a \emph{longest cycle transversal}.  The \emph{longest cycle transversal number} of a graph $G$, denoted $\lct(G)$, is the minimum size of a longest cycle transversal.  Unlike for longest paths, connectivity alone is not sufficient for the family of longest cycles in a graph to be pairwise intersecting, and $\lct(G)$ can be as large as $|V(G)|/3$ even when $G$ is connected. However, if $G$ is $2$-connected, then the longest cycles in $G$ are pairwise intersecting (see, for example, \Cref{lem:cycle_intersection}).  In many cases, bounds on $\lpt(G)$ for connected $G$ have analogous bounds on $\lct(G)$ for $2$-connected $G$.  Harvey and Payne~\cite{HP23} proved that $\lct(G) \le 2\ceil{\omega(G)/3}$ when $G$ is a $2$-connected chordal graph. Our second main result is the following.

\begin{theorem*}[See \Cref{thm:lct_bound}]
    If $G$ is an $n$-vertex $2$-connected chordal graph with minimal tree representation $T$, then $\lct(G)\le 4(1+\floor{\lg n})$.  
\end{theorem*}

In Section~\ref{sec:prelim}, we cover the necessary preliminaries to our results. In Section~\ref{sec:LCT}, we prove Theorem~\ref{thm:lct_bound} by an iterative divide-and-conquer strategy on the tree representation of $G$. Building on these techniques, in Section~\ref{sec:LPT} we prove Theorem~\ref{thm:lpt_bound}. In Section~\ref{sec:leafage}, we provide upper bounds for $\lpt(G)$ and $\lct(G)$ in terms of the leafage of $G$, a parameter introduced by Lin, McKee, and West \cite{LTW98}. As an application, we show the family of connected chordal graphs admitting a subdivided star tree representation is Gallai. In Section~\ref{sec:conclusion}, we conclude with some open questions.

\section{Preliminaries}\label{sec:prelim}

A \emph{tree representation} for a chordal graph $G$ is a \emph{host tree} $T$ along with a collection of subtrees ${\cal F}$ of $T$ such that the intersection graph on ${\cal F}$ is isomorphic to $G$. Recall that Gavril~\cite{GAVRIL74} characterized the chordal graphs as the graphs which admit a tree representation.  Although formally a tree representation is the intersection graph of subtrees of a host tree $T$, we also use $T$ to denote the tree representation.  For a vertex $u$ in a chordal graph $G$ and a tree representation $T$, we use $\subtree{u}$ to denote the subtree of $T$ corresponding to $u$.  When $H$ is a subgraph of $G$, we define $\subtree{H}$ to be the union, over all $u\in V(H)$, of $S(u)$.  Note that if $H$ is connected, then $\subtree{H}$ is a (connected) subtree of $T$. For each $x\in V(T)$, we define the \emph{bag} corresponding to $x$, denoted $\bag{x}$, to be the set $\{u\in V(G)\st x\in V(\subtree{u})\}$.  Note that for each $x\in V(T)$, the bag $\bag{x}$ is a clique in $G$.  Also, if $A$ is a clique in $G$, then the set $\{\subtree{u}\st u\in A\}$ is a family of pairwise intersecting subtrees of $T$.  Since subtrees of a tree have the Helly property, there is a vertex $x$ belonging to all of these subtrees, yielding $A\subseteq \bag{x}$.  Since $\bag{x}$ is always clique in $G$, it follows that if $A$ is a maximal clique in $G$, then $A=\bag{x}$ for some $x\in V(T)$.  In fact, as we note below, when $T$ is a minimal tree representation, the converse also holds and the bags in $T$ are exactly the maximal cliques in $G$.

A \emph{rooted tree} is a tree with a distinguished vertex $r$, called the \emph{root} of $T$.  If $T$ is a rooted tree with root $r$ and $u\in V(T)$, then the \emph{subtree of $T$ rooted at $u$} is the tree induced by the set of all vertices $w\in V(T)$ such that the $rw$-path in $T$ contains $u$.  Our convention is to use $X$ to name a rooted subtree of $T$ with root vertex $x$.  Often, our arguments produce a sequence of rooted subtrees $X_0, X_1,\ldots, X_t$ where $X_0$ is a tree representation of a chordal graph, each $X_i$ is a smaller rooted subtree of $X_{i-1}$ for $i\ge 1$, and the sequence ends with $X_t$ empty.

Let $T$ be a rooted tree, let $X$ be the subtree rooted at a vertex $x\in V(T)$, and let $Q$ be a subpath of $X$ with endpoint $x$.  For each $y\in V(Q)$, we define \emph{the descendants of $y$ in $X$ relative to $Q$}, denoted $\descendants{y}{Q}{X}$, to be the component of $X-E(Q)$ containing $y$.  For a subpath $Q_0$ of $Q$, we define $\descendants{Q_0}{Q}{X}$ to be the union, over $y\in V(Q_0)$, of $\descendants{y}{Q}{X}$.

A tree representation $T$ for a chordal graph $G$ is \emph{minimal} if there is no tree representation of $G$ with a host tree on fewer than $|V(T)|$ vertices.  If $T$ is a minimal tree representation for $G$, then for all distinct $x$ and $y$ in $V(T)$, we have $\bag{x}\not\subseteq \bag{y}$, as otherwise contracting the edge $e$ incident to $x$ along the $xy$-path simultaneously in $T$ and each subtree $\subtree{v}$ with $v\in V(G)$ and $e\in E(\subtree{v})$ yields a tree representation of $G$ with fewer vertices.  Note that if $\bag{x}$ were properly contained in a larger clique $A$ in $G$, then $A\subseteq \bag{y}$ for some $y\in V(T)$ and we have $\bag{x}\subsetneq A\subseteq \bag{y}$ for distinct $x,y\in V(T)$, contradicting the minimality of $T$. Conversely, for each tree representation $T$ of $G$, every clique in $G$ is contained in some bag of $T$ by the Helly property, and so the number of maximal cliques in $G$ is a lower bound on $|V(T)|$.  It follows that $T$ is minimal if and only if the bags of $T$ are the maximal cliques in $G$.  We often use the following property of minimal tree representations.
\begin{lemma}\label{lem:bigbags}
    Let $G$ be a connected chordal graph with minimal tree representation $T$ such that $|V(G)| \ge 2$.  We have $|\bag{y}| \ge 2$ for all $y \in V(T)$.
\end{lemma}
\begin{proof}
    As we have observed, in a minimal tree representation, each bag in $T$ is a maximal clique in $G$.  If $T$ has a bag of size $1$, then $G$ has a maximal clique of size $1$, which must be an isolated vertex.  Since $G$ is connected, it follows that $|V(G)| = 1$.
\end{proof}

Given a subgraph $H$ of a chordal graph $G$ with tree representation $T$, the \emph{core} of $H$ is the union, over all $uv\in E(H)$, of $V(\subtree{u}\cap \subtree{v})$. Note that the core of $H$ is a subset of $V(T)$. A vertex $v \in V(T)$ is a \emph{core vertex} of $H$ if $v$ belongs to the core of $H$. A set $W \subseteq V(T)$ has the \emph{core capture property} with respect to a family of subgraphs $\HH$ of $G$ if each $H \in \HH$ has a core which intersects $W$.  For a subgraph $X$ of $T$, we say that $X$ has the \emph{core capture property} with respect to $\HH$ if $V(X)$ has the core capture property with respect to $\HH$.

\begin{lemma}\label{lem:boundary-fence}
Let $G$ be a chordal graph with a tree representation $T$. Let $X$ be a rooted subtree of $T$ with root $x$, and let $H$ be a subgraph of $G$ such that $H$ has a core vertex in $X$ but $x$ is not a core vertex of $H$.  There is a vertex $w\in V(H)$ such that $S(w)\subseteq V(X)-x$.
\end{lemma}
\begin{proof}
    Let $y\in V(X)$ be a core vertex of $H$, and obtain $uv\in E(H)$ such that $y \in V(\subtree{u} \cap \subtree{v})$.  Note that for some $w\in\{u,v\}$, we have that $x\not\in V(\subtree{w})$, or else $x$ would also be a core vertex of $H$.  Since $\subtree{w}$ is a subtree of $T$, it must be that $\subtree{w}$ is contained in a component of $T-x$.  Since $y\in V(\subtree{w})$ and $y$ is in a component of $X-x$, it follows that $\subtree{w} \subseteq X-x$.
\end{proof}

  We frequently make use of Jordan's tree separator theorem \cite{J69}.

\begin{lemma}\label{lem:tree_separator}
    Let $T$ be a tree.  There exists a vertex $z \in V(T)$ such that each component of $T - z$ has at most $|V(T)|/2$ vertices.
\end{lemma}

A graph $G$ is \emph{$k$-connected} if $|V(G)|>k$ and $G-S$ is connected for all $S\subseteq V(G)$ with $|S| < k$.  The \emph{vertex connectivity} of $G$, denoted $\kappa(G)$, is the maximum $k$ such that $G$ is $k$-connected.  In a path $P$, the vertices of degree $2$ are \emph{interior} or \emph{internal} vertices and the vertices of degree less than $2$ are \emph{endpoints}.  Paths $P$ and $Q$ are \emph{internally disjoint} if their interior vertices are disjoint.  We use Menger's Theorem (see, e.g. \cite{W20}).

\begin{theorem}[Menger's Theorem]\label{thm:MengerSingle}
If $x$ and $y$ are distinct vertices in $G$, then $G$ has $\kappa(G)$ pairwise internally disjoint $xy$-paths.
\end{theorem}

For subsets $S,T \subseteq V(G)$, an $(S,T)$-path is a path with one endpoint in $S$ and the other endpoint in $T$ with internal vertices outside $S \cup T$.  Adding new vertices $s$ and $t$ to $G$ with $N(s) = S$ and $N(t)=T$ produces a new graph $G'$ with connectivity at least $\min\{|S|,|T|,\kappa(G)\}$.  Applying \Cref{thm:MengerSingle} to $G'$ and $s,t$, and shortening paths with internal vertices in $S\cup T$ gives the following well known consequence of Menger's Theorem.

\begin{theorem}[Menger's Theorem]\label{thm:MengerSets}
    Let $G$ be a graph. For all sets $S,T \subseteq V(G)$, there exists $\min\{|S|,|T|,\kappa(G)\}$ pairwise disjoint $(S,T)$-paths. 
\end{theorem}

\section{Longest Cycle Transversals}\label{sec:LCT}

In this section, we construct longest cycle transversals of size at most $O(\log n)$ in $2$-connected chordal graphs.  We first prove a folklore lemma.

\begin{lemma}\label{lem:cycle_intersection}
    If $G$ is 2-connected graph and $C_1$ and $C_2$ are longest cycles in $G$, then $C_1 \cup C_2$ is $2$-connected.  In particular, $|V(C_1)\cap V(C_2)|\ge 2$.
\end{lemma}
\begin{proof}
    Let $H=C_1\cup C_2$.  Note that $V(C_1)$ and $V(C_2)$ have nonempty intersection or else applying \Cref{thm:MengerSets} with $(S,T)=(V(C_1), V(C_2))$ gives two disjoint paths in $G$ joining $C_1$ and $C_2$ and hence a longer cycle.  We show that if $V(C_1) \cap V(C_2)$ contains a single vertex $z$, then $G$ has a longer cycle.  Since $G-z$ is connected, there is a path $P$ in $G$ joining $V(C_1 - z)$ and $V(C_2 - z)$.  Let $x$ be the endpoint of $P$ in $C_1 - z$ and let $y$ be the endpoint of $P$ in $C_2 - z$.  We obtain a longer cycle by combining $P$ with the longer $xz$-subpath of $C_1$ and the longer $yz$-subpath of $C_2$.  Therefore $|V(C_1) \cap V(C_2)| \ge 2$.  Since $C_1$ and $C_2$ are $2$-connected, it follows that $H$ is also $2$-connected.  
\end{proof}

Let $G$ be a chordal graph with tree representation $T$, and let $P$ be a path in $T$.  The components of $T-E(P)$ containing an endpoint of $P$ are \emph{exterior components}, and the components containing an interior vertex of $P$ are \emph{interior components}.  Note that if $P$ has no interior vertices, then $T-E(P)$ has no interior components and either one or two exterior components according to whether $|V(P)| = 1$ or $|V(P)| = 2$.

\begin{lemma}\label{lem:spanvertex}
Let $G$ be a chordal graph with a tree representation $T$, let $x$ and $y$ be distinct vertices in $T$, and let $P$ be the $xy$-path in $T$.  If $G$ has a subgraph $H$ such that $H$ has a core vertex in each exterior component of $T-E(P)$ but no core vertex in any interior component, then $H$ contains $\kappa(H)$ vertices $v$ such that $S(v)$ contains $P$. 
\end{lemma}

\begin{proof}
Let $T_x$ and $T_y$ be the exterior components of $T-E(P)$ containing $x$ and $y$ respectively.  Let $e_x$ and $e_y$ be edges in $H$ having a core vertex in $T_x$ and $T_y$ respectively.  Both endpoints of $e_x$ have subtrees intersecting $T_x$, and since $H$ has no core vertex in an interior component of $T-E(P)$, at least one of these endpoints $u_x$ has a subtree $S(u_x)$ that is contained in $T_x$.  Similarly, let $u_y$ be an endpoint of $e_y$ with $S(u_y)\subseteq T_y$.  Let $k=\kappa(H)$.  By \Cref{thm:MengerSingle}, it follows that $H$ has $k$ internally disjoint $u_xu_y$-paths $Q_1,\ldots,Q_k$.

We claim that each $Q_i$ contains a vertex $v_i$ with $P\subseteq S(v_i)$.  If not, then each vertex $v$ in $Q_i$ has a subtree $S(v)$ that is either disjoint from $T_y$ or disjoint from $T_x$. Since $Q_i$ has endpoints $u_x$ and $u_y$ which are disjoint from $T_y$ and $T_x$ respectively, it follows that $Q_i$ has adjacent vertices $ww'$ such that $S(w)$ is disjoint from $T_y$ and $S(w')$ is disjoint from $T_x$.  It follows that $S(w)\cap S(w')$ is disjoint from $T_x \cup T_y$, and so $S(w)\cap S(w')$ has all its vertices contained in the interior components of $T-E(P)$, contradicting that the core of $H$ is disjoint from those components.
\end{proof}

A \emph{rooted tree representation} of a chordal graph $G$ is a tree representation $T$ such that $T$ is a rooted tree.  Let $G$ be a chordal graph with a rooted tree representation $T$, let $X$ be the subtree of $T$ rooted at $x$, and suppose that $X$ has the core capture property for a collection $\HH$ of longest paths or longest cycles in $G$.  We often make progress by finding a proper subtree $X'$ of $X$ and a small set of vertices $A$ in $G$ such that for each $H\in\HH$, either $H$ intersects $A$ or $H$ has a core vertex in $X'$.  Let $z\in V(X)$ and let $Q$ be the $xz$-path in $X$.  Recall that $\descendants{Q}{Q}{X}$ equals the subgraph $X-E(Q)$, and therefore $X$ and $\descendants{Q}{Q}{X}$ have the same vertex set.  It follows that $\descendants{Q}{Q}{X}$ also has the core capture property for $\HH$.  It turns out that looking at a minimal subpath $Q_0$ of $Q$ such that $\descendants{Q_0}{Q}{X}$ has the core capture property for $\HH$ is useful in constructing $A$.

\begin{lemma}\label{lem:spanlongpath}
    Let $G$ be a chordal graph with rooted tree representation $T$, let $X$ be a subtree of $T$ with root $x$, and let $Q$ be a subpath of $X$ such that $x$ is an endpoint of $Q$.  Let $\HH$ be a nonempty family of subgraphs of $G$ such that $X$ has the core capture property for $\HH$ and let $k=\min\{\kappa(H_1\cup H_2)\st H_1,H_2\in\HH\}$.  Let $Q_0$ be a minimal subpath of $Q$ such that $\descendants{Q_0}{Q}{X}$ has the core capture property for $\HH$.  If $|V(Q_0)|\ge 2$, then $G$ has $k$ vertices $v$ such that $S(v)$ contains $Q_0$.
\end{lemma}
\begin{proof}
Let $y_1$ and $y_2$ be the endpoints of $Q_0$.  By minimality of $Q_0$, for $i\in \{1,2\}$, there exists $H_i\in \HH$ such that $H_i$ has a core vertex in $\descendants{y_i}{Q}{X}$ but no core vertex in $\descendants{Q_0 - y_i}{Q}{X}$. Let $H=H_1 \cup H_2$ and note that $H$ has a core vertex in each exterior component of $X-E(Q_0)$ but no core vertex in an interior component of $T-E(Q_0)$.  It follows from~\Cref{lem:spanvertex} that $G$ has $\kappa(H)$ vertices $v$ such that $S(v)$ contains $Q_0$, and $\kappa(H)\ge k$.
\end{proof}

In our lemma below, we make use of a pair of vertices $W$ which have appropriate neighbors for detouring segments of certain paths to obtain longer paths.  For convenience, we refer to the vertices in $W$ as \emph{glue vertices}.  An \emph{attachment point} is a vertex $u\in V(G)$ such that $W\subseteq N(u)$. 

\begin{lemma}\label{lem:general-two-way}
    Let $G$ be a graph, let $W \subseteq V(G)$ be a pair of  vertices, and let $\HH$ be a family of subgraphs of $G$, where each $H\in\HH$ is a longest path or longest cycle in $G$ and $V(H)\cap W=\emptyset$.  Let $\RR$ be a nonempty family of paths in $G$ such that each $R\in \RR$ is disjoint from $W$ and the endpoints of $R$ are distinct attachment points in $G$.  If each $H\in \HH$ contains a subpath in $\RR$, then each longest path in $\RR$ intersects every $H$ in $\HH$.
\end{lemma}
\begin{proof}
Let $R$ be a longest path in $\RR$, and suppose for a contradiction that $R$ is disjoint from some $H\in\HH$.  Let $R_0$ be a subpath of $H$ in $\RR$, and let $x$ and $y$ be the endpoints of $R_0$.  We have $|V(R_0)| \le |V(R)|$.  Let $W=\{w_1,w_2\}$.  We modify $H$ to obtain a longer path or cycle in $G$ by replacing the subpath $R_0$ with the path $xw_1Rw_2y$. 
\end{proof}

We are now ready to prove the main technical lemma for finding a small longest cycle transversal in $2$-connected chordal graphs.

\begin{lemma}\label{lem:cycle-trans}
    Let $G$ be a $2$-connected chordal graph with minimal rooted tree representation $T$.  Let $X$ be a subtree of $T$ rooted at $x$, let $z\in V(X)$, and let $Q$ be the $xz$-path in $X$.  Let $\CC$ be a family of longest cycles in $G$ such that $X$ has the core capture property for $\CC$.  There is a set $A \subseteq V(G)$ with $|A| \le 4$ such that the family $\CC'$ of cycles in $\CC$ that are disjoint from $A$ is either empty, or there is a component $X'$ of $X - V(Q)$ such that $X'$ has the core capture property for $\CC'$.
\end{lemma}

\begin{proof}
    Note that if $\CC = \emptyset$, then the lemma is trivially satisfied with $A=\emptyset$.  So we assume that $\CC$ contains a cycle, implying $|V(G)| \ge 3$.  It follows from \Cref{lem:bigbags} that $|\bag{y}| \ge 2$ for each $y\in V(T)$.
        
    Let $z \in V(X)$ and let $Q$ be the $xz$-path in $X$.  Note that $V(\descendants{Q}{Q}{X}) = V(X)$, and hence $\descendants{Q}{Q}{X}$ has the core capture property for $\CC$.  Let $Q_0$ be a minimal subpath of $Q$ such that $\descendants{Q_0}{Q}{X}$ also has the core capture property for $\CC$.  We claim that $G$ has a pair of vertices $\{w_1,w_2\}$ with each $\subtree{w_i}$ containing $Q_0$.  If $|V(Q_0)|\ge 2$, then this follows from \Cref{lem:spanlongpath} since by \Cref{lem:cycle_intersection}, we have that $\kappa(C_1\cup C_2)\ge 2$ for all $C_1,C_2\in\CC$.  Otherwise, if $Q_0$ consists of a single vertex $y$, then we choose $\{w_1,w_2\}$ to be a pair of vertices from $\bag{y}$ arbitrarily.  Let $W=\{w_1,w_2\}$.

    Let $\CC_1$ be the set of all cycles $C\in\CC$ such that $C$ is disjoint from $W$. We may suppose $\CC_1\ne\emptyset$, or else the lemma is satisfied with $A=W$.  We claim that each $C\in\CC_1$ contains a vertex whose subtree is contained in a component of $\descendants{Q_0}{Q}{X} - V(Q_0)$.  Since $C$ has a core vertex in $\descendants{Q_0}{Q}{X}$, it follows that $C$ contains adjacent vertices $u_1$ and $u_2$ such that $\subtree{u_1}\cap\subtree{u_2}$ intersects $\descendants{Q_0}{Q}{X}$.  Since $C$ is a longest cycle and $w_1\not\in V(C)$, at least one of $\{\subtree{u_1},\subtree{u_2}\}$ is disjoint from $V(Q_0)$, or else $C$ would extend to a longer cycle by inserting $w_1$ between $u_1$ and $u_2$.  Hence one of $\{\subtree{u_1},\subtree{u_2}\}$ is contained in a component of $\descendants{Q_0}{Q}{X}-V(Q_0)$.  
    
    Suppose that there exists $C\in\CC_1$ and a component $Y$ of $\descendants{Q_0}{Q}{X} - V(Q_0)$ such that all but at most one vertex $v\in V(C)$ satisfies $\subtree{v} \subseteq Y$.  It follows from \Cref{lem:cycle_intersection} that each cycle in $\CC_1$ intersects $C$ in at least one vertex whose subtree is contained in $Y$.  Therefore each cycle in $\CC_1$ has a vertex whose subtree is contained in $Y$ and it follows that $Y$ has the core capture property for $\CC_1$.  Hence the lemma is satisfied with $A=W$ and $X'=Y$.  So we may assume that each $C\in\CC_1$ contains distinct vertices $u,v_1,v_2$ such that $\subtree{u}\subseteq Y$ for some component $Y$ of $\descendants{Q_0}{Q}{X}-V(Q_0)$ and $\subtree{v_i}$ intersects $Q_0$ for $i\in\{1,2\}$.

    Let $\RR$ be the family of paths $R$ in $G-W$ such that $|V(R)| \ge 3$, the endpoints of $R$ have subtrees intersecting $Q_0$, and each interior vertex $v$ of $R$ satisfies $\subtree{v}\subseteq Y$ for some component $Y$ of $\descendants{Q_0}{Q}{X} - V(Q_0)$.  For each $C\in\CC_1$, a subpath of $C$ which is maximal subject to internal vertices having subtrees contained in a component of $\descendants{Q_0}{Q}{X}-V(Q_0)$ is in the family $\RR$.  Since $\CC_1$ is nonempty, so is $\RR$.  Let $R$ be a longest path in $\RR$.  It follows from \Cref{lem:general-two-way} with $W$ playing the role of glue vertices that each $C\in\CC_1$ intersects $R$.  Let $w_3$ and $w_4$ be the endpoints of $R$, and let $\CC_2$ be the set of all $C\in\CC_1$ that are disjoint from $\{w_3,w_4\}$.  Let $Y$ be the component of $\descendants{Q_0}{Q}{X}-V(Q_0)$ such that each internal vertex $v$ of $R$ satisfies $\subtree{v}\subseteq Y$.  Note that $Y$ is also a component of $\descendants{Q}{Q}{X} - V(Q)$. Furthermore, each $C\in\CC_2$ must intersect $R$ in an interior vertex of $R$ and hence each $C\in\CC_2$ has a core vertex in $Y$.  We set $A=\{w_1,w_2,w_3,w_4\}$ and $X'=Y$.  Since each $C\in\CC$ either intersects $A$ or has a core vertex in $Y$, the lemma is satisfied.  
\end{proof}

For a $2$-connected chordal graph $G$ with host tree $T$, we give an upper bound on $\lct(G)$ in terms of the value of a 2-player \emph{Cutter-Chooser game}.  In the game, the play by Cutter represents our freedom in constructing the transversal to divide the host tree $T$ in an advantageous way, and the play by Chooser represents the worst-case scenario for the location of longest cycles not yet covered by our transversal.  The name is borrowed from the classic cake-cutting literature, see for example \cite{BT96}. The game proceeds in rounds, as follows.  

At the start of each round, Cutter is presented with a rooted subtree $X$ of $T$ with root vertex $x$; in the initial round, $X=T$ and $x$ is the root of $T$.  Cutter selects a vertex $z\in V(X)$, and the subtree $X$ is then cut along the $xz$-path $Q$.  The game ends if $X=Q$.  Otherwise, Chooser responds by selecting a component $X'$ of $X-V(Q)$.  Note that $X'$ is a rooted subtree of $X$.  Let $x'$ be the root of $X'$.  The subtree $X'$ with root $x'$ is presented to Cutter at the start of the next round.  Cutter tries to minimize the number of rounds, and Chooser tries to maximize the number of rounds.  Let $\ccg{T,r}$ be the number of rounds in the Cutter-Chooser game on $T$ with root $r$ under optimal play.  For a (non-rooted) tree $T$, we define $\ccg{T} = \min_{r\in V(T)} \ccg{T,r}$.  One may view the parameter $\ccg{T}$ as a slight variant of the Cutter-Chooser game in which Cutter may choose the root vertex $r$ before the initial round.

In a graph $G$, the \emph{subdivision of an edge} $uv$ replaces $uv$ with a path of length $2$ through a new vertex.  We say that $G'$ is a \emph{subdivision} of $G$ if $G'$ can be obtained from $G$ by a sequence of zero or more edge subdivisions.  Note that $\ccg{T} = 1$ if and only if $T$ is path.  A \emph{caterpillar} is a tree $T$ containing a \emph{spine subpath} $Q$ such that each leaf in $T$ has its neighbor on $Q$.  If $T$ is a subdivision of a caterpillar, then Cutter may select the spine path in the first round, forcing Chooser to select some path to start the second round.  Cutter then cuts along this path, ending the game.  Hence, if $T$ is a subdivided caterpillar, then $\ccg{T}\le 2$.  The converse also holds, and so $\ccg{T}\le 2$ if and only if $T$ is a subdivided caterpillar.  We will need a simple monotonicity property for $\ccg{T,r}$.

\begin{prop}\label{prop:ccg-mono}
Let $T$ be a tree with root $r$, and let $Y$ be the subtree of $T$ rooted at $y$.  We have $\ccg{Y,y}\le \ccg{T,r}$.
\end{prop}
\begin{proof}
We show that $\ccg{Y,y} \le \ccg{T,r}$ by giving a strategy for Chooser in the Cutter--Chooser game on $T$ with root $r$.  The strategy has two phases.  In the first phase, the rounds begin with a subtree $X$ of $T$ with root $x$ such that $Y\subseteq X$.  Let $z$ be the vertex selected by Cutter, and let $Q$ be the $xz$-path in $X$.  If $Q$ and $Y$ are disjoint, then Chooser selects the component of $X-V(Q)$ that contains $Y$ and the first phase continues.  Otherwise $Q$ intersects $Y$ and the second phase has begun.  Chooser simulates a Cutter--Chooser game on $Y$ with root $y$, where Cutter initially cuts along the $yz$-subpath of $Q$, after which Chooser uses an optimal strategy on $Y$.  Since Chooser forces at least $\ccg{Y,y}$ rounds in the second phase, the lower bound on $\ccg{T,r}$ follows.
\end{proof}

Our next proposition uses a divide and conquer strategy for Cutter to show that $\ccg{T} = O(\log n)$ when $T$ has $n$ vertices.

\begin{proposition}\label{prop:ccg-bound}
    If $T$ is an $n$-vertex tree with root $r$, then $\ccg{T}\le \ccg{T,r}\le 1+\floor{\lg n}$.
\end{proposition}
\begin{proof}
    We give a strategy for Cutter that shows that $\ccg{T,r}\le 1+\floor{\lg n}$.  Initially, we set $X_0=T$.  For $i\ge 1$, at the start of the $i$th round, Cutter uses \Cref{lem:tree_separator} to pick a vertex $z\in V(X_{i-1})$ such that each component of $X_{i-1}-z$ has size at most $|V(X_{i-1})|/2$.  Let $Q$ be the $xz$-path in $X_{i-1}$, and let $X_i$ be the component of $X-V(Q)$ which Chooser selects.  The game ends in round $t$, when Chooser has no remaining components to choose. Since $|V(X_i)| \le |V(X_{i-1})|/2$ for $i\ge 1$, we have that $|V(X_i)| \le n/2^{i}$.  It follows that the total number of rounds $t$ is at most $1+\floor{\lg n}$.
\end{proof}

When $T$ is an $n$-vertex balanced binary tree, it is easy to see that $\ccg{T}=\Theta(\log n)$, and so our upper bound on $\ccg{T}$ in \Cref{prop:ccg-bound} is sharp up to multiplicative constants.  By using an optimal strategy for Cutter to iterate \Cref{lem:cycle-trans}, we obtain an upper bound on $\lct(G)$.

\begin{theorem}\label{thm:lct-upper-game}
    If $G$ is a $2$-connected chordal graph with minimal tree representation $T$, then $\lct(G)\le 4(\ccg{T})$.
\end{theorem}
\begin{proof}
Let $x$ be the initial root vertex in $V(T)$ chosen by Cutter under optimal play. Let $\CC$ be the family of longest cycles in $G$.  We initially set $(\CC_0,X_0) = (\CC,T)$.  We apply \Cref{lem:cycle-trans} once per round of a simulated Cutter-Chooser game to obtain $(\CC_0,X_0),\ldots,(\CC_t,X_t)$ and $A_1,\ldots,A_t$ such that $\CC = \CC_0 \supseteq \cdots \supseteq \CC_t$, $T=X_0 \supseteq \cdots \supseteq X_t$, the subtree $X_i$ has the core capture property for $\CC_i$ for $i \ge 0$, and each cycle in $\CC_{i-1} \setminus \CC_{i}$ intersects $A_i$ for all $i \ge 1$.  In our simulated game, we use an optimal strategy for Cutter, and we have Chooser play according to the choices made by \Cref{lem:cycle-trans}. 

Let $i\ge 1$.  Suppose that $X_{i-1}$ has the core capture property for $\CC_{i-1}$ and that our simulated game is at the start of stage $i$.  Cutter is presented with the rooted subtree $X_{i-1}$.  Let $x_{i-1}$ be the root of $X_{i-1}$.  Let $z\in V(X_{i-1})$ be the vertex chosen by Cutter in the simulated Cutter-Chooser game, and let $Q$ be the $x_{i-1}z$-path in $X_{i-1}$.  By \Cref{lem:cycle-trans}, we obtain a set of vertices $A_i\subseteq V(G)$ with $|A_i|\le 4$ such that the family $\CC_i$ of cycles in $\CC_{i-1}$ that are disjoint from $A_i$ is either empty, or there exists a subtree $X_i$ of $X_{i-1} - V(Q)$ such that $X_i$ has the core capture property for $\CC_i$.  If $\CC_i = \emptyset$, then the iteration ends with $t=i$ (the choice of $X_t$ is arbitrary).  Otherwise, we update our simulated game to have Chooser pick $X_i$ and we proceed to stage $i+1$ with the pair $(\CC_i, X_i)$.

The iteration ends with a pair $(\CC_t,X_t)$ such that $\CC_t$ is empty.  Let $A=\bigcup_{i=1}^t A_i$, and note that $A$ is a longest cycle transversal for $G$ since each $C\in\CC$ intersects $A_i$, where $i$ is the unique index such that $C\in \CC_{i-1} - \CC_i$.  Since each $A_i$ has size at most $4$ and $t\le\ccg{T}$, the bound follows.
\end{proof}

We are now prepared to prove our main result on longest cycle transversals. 

\begin{theorem}\label{thm:lct_bound}
If $G$ is an $n$-vertex $2$-connected chordal graph, then $\lct(G)\le 4(1+\floor{\lg n})$. 
\end{theorem}
\begin{proof}
    Since $G$ is an $n$-vertex chordal graph, $G$ has at most $n$ maximal cliques (see, for example, \cite{G72}).  Let $T$ be a minimal tree representation for $G$.  Since bags in $T$ correspond to maximal cliques in $G$ (see the discussion preceding \Cref{lem:bigbags}), it follows that $|V(T)|\le n$.  By \Cref{prop:ccg-bound} and \Cref{thm:lct-upper-game}, we have $\lct(G)\le 4(\ccg{T}) \le 4(1+\floor{\lg n})$. 
\end{proof}

A more careful analysis shows that $\lct(G)\le 4(\ccg{T}-1)$ when $G$ is a $2$-connected chordal graph with a host tree $T$ satisfying $\ccg{T}>1$.  We sketch the details.  If the iteration lasts at most $\ccg{T}-1$ rounds, then the bound is clear.  Suppose that $t=\ccg{T}$.  At the start of round $t-1$, when Cutter is presented with $(\CC_{t-2}, X_{t-2})$, Cutter selects a path $Q$ in $X_{t-2}$ and obtains a set $A_{t-1}$ of at most $4$ vertices and Chooser selects a component $X_{t-1}$ of $X_{t-2} - V(Q)$ such that $X_{t-1}$ has the core capture property for $\CC_{t-1}$, where $\CC_{t-1}$ is the set of cycles in $\CC_{t-2}$ that are disjoint from $A_{t-1}$. Note that $X_{t-1}$ is necessarily a path, as the game terminates on round $t$ and $t=\ccg{T}$. If some cycle $C\in\CC_{t-1}$ satisfies $\subtree{v}\subseteq X_{t-1}$ for all but at most one vertex $v\in V(C)$, then every longest cycle in $G$ intersects $C$ at least twice and hence $X_{t-1}$ has the core capture property for all of $\CC$, let alone $\CC_{t-2}$.  We may then take $A_{t-1} = \emptyset$ and $\CC_{t-1} = \CC_{t-2}$.  Otherwise, since $X_{t-1}$ is a path, we may choose $w_3$ and $w_4$ in \Cref{lem:cycle-trans} such that $\subtree{w_3}$ and $\subtree{w_4}$ extend maximally into $X_{t-1}$ subject to intersecting $Q$.  In this case, $A_{t-1}$ intersects every cycle in $\CC_{t-2}$ and so $\CC_{t-1} = \emptyset$.  With $\CC_{t-1}$ already empty, it is unnecessary to include $A_t$ in our longest cycle transversal.  In both cases, we save $4$ vertices in our transversal.  It follows that $\lct(G)\le 4$ when $G$ is a $2$-connected chordal graph admitting a subdivided caterpillar host tree.  It would be interesting to obtain better upper bounds on $\lct(G)$ for $2$-connected chordal graphs that admit subdivided caterpillar host trees.

\section{Longest Path Transversals}\label{sec:LPT}

In this section, we prove that each connected chordal graph $G$ with $n$ vertices has a longest path transversal of size $O(\log^2 n)$.  The main ideas are similar to those in \Cref{sec:LCT}, but the asymmetry of the endpoints of paths presents some additional complications.  

Suppose that $G$ is a chordal graph with minimal rooted tree representation $T$, let $X$ be a subtree of $T$ rooted at $x$, and let $\PP$ be a family of longest paths such that $X$ has the core capture property for $\PP$.  A $uv$-path $P\in \PP$ is a \emph{round trip path} if both $\subtree{u}$ and $\subtree{v}$ have a vertex outside $X$.  Although $P$ may have interior vertices with subtrees contained deep inside $X$, we have that $P$ starts and ends with subtrees that are not contained in $X$.  A $uv$-path $P\in \PP$ is a \emph{one way path} if at least one of $\subtree{u}$ and $\subtree{v}$ is contained in $X$.  Let $\PP_1$ and $\PP_2$ respectively be the set of round trip paths and one way paths in $\PP$.  As with longest cycle transversals, we would like to find a small set $A$ of vertices in $G$ such that the family $\PP'$ of paths in $\PP$ that are disjoint from $A$ is either empty or there is a proper subtree $X'$ of $X$ having the core capture property for $\PP'$.  Unfortunately, we may need distinct subtrees $X'_1$ and $X'_2$ such that $X'_{\ell}$ has the core capture property for $\PP' \cap \PP_\ell$ for $\ell\in\{1,2\}$.  Since round trip paths with respect to $X$ are also round trip paths with respect to subtrees of $X$, we find transversals for the round trip paths first.  

\begin{lemma}\label{lem:roundpath}
Let $G$ be a connected chordal graph with minimal rooted tree representation $T$.  Let $X$ be a subtree of $T$ rooted at $x$, let $z\in V(X)$, and let $Q$ be the $xz$-path in $X$.  Let $\PP$ be a family of longest paths in $G$ such that $X$ has the core capture property for $\PP$ and the endpoints of each $P\in\PP$ have subtrees that contain a vertex outside $X$.  There is a set $A\subseteq V(G)$ with $|A| \le 4$ such that the family $\PP'$ of paths in $\PP$ that are disjoint from $A$ is either empty, or there is a component $X'$ of $X-V(Q)$ such that $X'$ has the core capture property for $\PP'$.
\end{lemma}

\begin{proof}
If $|V(G)| = 1$, then we may take $A=V(G)$ with $\PP'=\emptyset$.  Hence by \Cref{lem:bigbags}, we may assume that each bag in $T$ has size at least $2$. 

Let $z$ be a vertex in $X$ and let $Q$ be the $xz$-path in $X$.  Recall that $V(\descendants{Q}{Q}{X} ) = V(X)$, and so $\descendants{Q}{Q}{X}$ has the core capture property for $\PP$.  Let $Q_1$ be a minimal subpath of $Q$ such that $\descendants{Q_1}{Q}{X}$ has the core capture property for $\PP$.  We claim that $G$ has a vertex $w_1$ such that $S(w_1)$ spans $Q_1$.  If $|V(Q_1)|=1$ and $Q_1 = y$, then we may take $w_1$ to be any vertex in $B(y)$.  Otherwise, if $|V(Q_1)|\ge 2$, then we apply \Cref{lem:spanlongpath} to obtain $w_1$.

Let $\PP_1$ be the set of all paths $P\in \PP$ that do not contain $w_1$; note that if $\PP_1$ is empty, then the lemma is satisfied with $A=\{w_1\}$ and $\PP'$ empty.  So assume that $\PP_1$ is nonempty and let $Q_2$ be a minimal subpath of $Q_1$ such that $\descendants{Q_2}{Q}{X}$ has the core capture property with respect to $\PP_1$.  Our next aim is to obtain $w_2\in V(G)$ such that $w_2\ne w_1$ and $S(w_2)$ spans $Q_2$.  Indeed, if $|V(Q_2)| = 1$ with $Q_2 = y$, then we may use $|B(y)|\ge 2$ to choose $w_2\in B(y)$ distinct from $w_1$.  Otherwise, if $|V(Q_2)| \ge 2$, then we obtain $w_2$ by applying \Cref{lem:spanlongpath} with $\HH=\PP_1$, and observing that $w_2$ is chosen from the union of two paths in $\PP_1$, neither of which contains $w_1$.

Let $\PP_2$ be the set of all paths in $\PP_1$ which do not contain $w_2$.  Again, we may assume $\PP_2$ is nonempty, or else the lemma is satisfied with $A=\{w_1,w_2\}$ and $\PP'$ empty.  Let $\RR$ be the family of paths in $G - \{w_1,w_2\}$ of size at least $3$ whose endpoints have subtrees intersecting $Q_2$ and whose interior vertices $u$ satisfy $S(u)\subseteq \descendants{Q_2}{Q}{X}- V(Q_2)$.  We claim that each $P\in\PP_2$ has a subpath in $\RR$.  Indeed, since $\descendants{Q_2}{Q}{X}$ has the core capture property for $\PP_1$ and $P\in \PP_2\subseteq \PP_1$, it follows that $P$ has a core vertex in $\descendants{Q_2}{Q}{X}$.  Since $P$ is a longest path, $V(Q_2)\subseteq S(w_2)$, and $w_2\not\in V(P)$, it follows that $P$ has no core vertex in $V(Q_2)$, as if $y\in V(Q_2)$ and $y\in S(u)\cap S(v)$ for some $uv\in E(P)$, then we obtain a longer path by inserting $w_2$ between $u$ and $v$.  Since $P$ has a core vertex in $\descendants{Q_2}{Q}{X}$ but no core vertex in $V(Q_2)$, it follows that there exists $u\in V(P)$ with $S(u)\subseteq \descendants{Q_2}{Q}{X}-V(Q_2)$.  Since the endpoints of $P$ have subtrees intersecting $T-V(X)$, it follows that $u$ is an interior vertex in a subpath of $P$ in $\RR$.  

Let $W=\{w_1,w_2\}$, and let $R$ be a path in $\RR$ of maximum length.  Applying \Cref{lem:general-two-way} with $\HH=\PP_2$ to $W$ and $\RR$, it follows that $V(R)$ intersects each path in $\PP_2$.  Let $w_3$ and $w_4$ be the endpoints of $R$.  Note that $R-\{w_3,w_4\}$ is a nonempty connected subgraph of $G$, all of whose subtrees are contained in a single component of $\descendants{Q_2}{Q}{X} - V(Q_2)$, which happens to be a rooted subtree $X'$ of $X$.

Let $A=\{w_1,w_2,w_3,w_4\}$.  Let $P$ be a path in $\PP$ disjoint from $A$, and note that $P\in \PP_2$.  Recall that $P$ and $R$ intersect, and since the endpoints of $R$ are contained in $A$, it follows that $P$ contains a vertex $u$ in the interior of $R$.  Since $S(u)\subseteq X'$ and since $|V(P)| > 1$, it follows that $P$ has a core vertex in $X'$.  
\end{proof}

We now prove a corollary analogous to Theorem \ref{thm:lct-upper-game} but for Lemma \ref{lem:roundpath}.

\begin{cor}\label{cor:roundpath}
Let $G$ be a connected chordal graph with minimal rooted tree representation $T$, and let $X$ be a rooted subtree of $T$ with root $x$. Let $\PP$ be a family of longest paths in $G$ such that $X$ has the core capture property for $\PP$ and the endpoints of each $P\in\PP$ have subtrees that contain a vertex outside $X$.  There is a set $A\subseteq V(G)$ that intersects each path in $\PP$ with $|A| \le 4(\ccg{X,x})$.
\end{cor}
\begin{proof}
We simulate a Cutter-Chooser game on $X$ with root $x$.  In our simulation, we use an optimal strategy for Cutter and we have \Cref{lem:roundpath} make decisions for Chooser.  

Let $(\PP_0,X_0) = (\PP,X)$.  We iteratively apply \Cref{lem:roundpath} to obtain a sequence $(\PP_0, X_0), \ldots, (\PP_t,X_t)$ with $\PP = \PP_0 \supseteq \ldots \supseteq \PP_t$ and $X=X_0 \supseteq \ldots \supseteq X_t$ and sets $A_1,\ldots,A_t$ such that $X_i$ has the core capture property for $\PP_i$ for $i \ge 0$, each path $P\in \PP_{i-1} \setminus \PP_i$ intersects $A_i$, and $t\le \ccg{X,x}$. 

Let $i\ge 1$.  At the start of stage $i$ in our simulated game, Cutter is presented with the subtree $X_{i-1}$.  Let $x_{i-1}$ be the root of $X_{i-1}$, let $z\in V(X_{i-1})$ be the vertex selected by the optimal strategy for Cutter, and let $Q$ be the $xz$-path in $X_{i-1}$.  By \Cref{lem:roundpath}, there exists a set $A_i\subseteq V(G)$ with $|A_i| \le 4$ such that the set of paths $\PP_i$ which are disjoint from $A_i$ is either empty or some component $X_i$ of $X_{i-1} - V(Q)$ has the core capture property for $\PP_i$.  If $\PP_i$ is empty, then we set $t=i$ and the iteration ends (and the choice of $X_t$ is arbitrary). Otherwise, Chooser selects the component $X_i$ and the game proceeds to the next round. Since Cutter plays optimally and the simulated game ends after $t$ rounds, we have that $t\le \ccg{X,x}$ and $\PP_t=\emptyset$.  Therefore $\bigcup_{i=1}^t A_i$ is a transversal for $\PP$ of size at most $4t$.
\end{proof}

\Cref{lem:general-two-way} considers a set of paths $\RR$ such that each $R\in\RR$ has endpoints that are adjacent to the glue vertices; these paths start in and return to the neighborhood of the glue vertices.  Our next lemma is an analogue for ``one way'' paths that start in the neighborhood of a glue vertex and need not return.  A \emph{suffix} of a path $P$ is a subpath of $P$ containing an endpoint of $P$.

\begin{lemma}\label{lem:one-way}
    Let $G$ be a connected graph, let $w$ be a glue vertex, and let $\PP$ be a family of longest paths in $G$ with each $P\in\PP$ avoiding $w$.  An attachment point is a vertex $u\in N(w)$.  Let $\RR$ be a nonempty family of paths in $G$ such that each $R\in\RR$ avoids $w$ and has at least one attachment endpoint.  If every path $P\in\PP$ contains a suffix in $\RR$, then a longest path in $\RR$ intersects every path in $\PP$.
\end{lemma}
\begin{proof}
    Let $R$ be a longest path in $\RR$ and suppose that some $P\in\PP$ is disjoint from $R$.  Let $R_0$ be a suffix of $P$ with $R_0\in\RR$.  Let $x$ and $y$ be the endpoints of $R_0$, with $y$ also serving as an endpoint of $P$.  Note that $y$ is not an attachment point, or else we may extend $P$ by appending $w$ at $y$.  Therefore $x$ is an attachment point.  Since $|V(R)| \ge |V(R_0)|$, we obtain a longer path by replacing $R_0$ with the path $xwR$ (with $R$ oriented appropriately).
\end{proof}

Our next lemma applies \Cref{cor:roundpath} to find a transversal for round trip paths; the one way paths are treated by adding a small set of vertices to the transversal and obtaining a subtree $X'$ with the core capture property for the remaining one way paths.   

\begin{lemma}\label{lem:allpaths}
Let $G$ be a connected chordal graph with minimal rooted tree representation $T$. Let $X$ be a subtree of $T$ rooted at $x$, let $z\in V(X)$, and let $Q$ be the $xz$-path in $X$. Let $\PP$ be a family of longest paths in $G$ such that $X$ has the core capture property for $\PP$. There is a set $A\subseteq V(G)$ with $|A| \le 4(\ccg{X,x}) + 5$ such that the family $\PP'$ of paths in $\PP$ that are disjoint from $A$ is either empty, or there is a component $X'$ of $X-V(Q)$ such that $X'$ has the core capture property for $\PP'$. 
\end{lemma}

\begin{proof}
If $|V(G)|=1$, then we may take $A=V(G)$ with $\PP'$ empty.  Hence, by \Cref{lem:bigbags}, we may assume each bag in $T$ has size at least $2$.  If $\PP$ is empty, then we may take $A=\emptyset$ with $\PP'$ also empty.

Let $Q_1$ be a minimal subpath of $Q$ such that $\descendants{Q_1}{Q}{X}$ has the core capture property for $\PP$.  We claim some vertex $w_1\in V(G)$ has a subtree $S(w_1)$ that contains $Q_1$.  If $Q_1$ is a single vertex $y$, then we may take $w_1$ to be any vertex in $B(y)$.  Otherwise, we apply \Cref{lem:spanlongpath} with $\HH=\PP$ to obtain $w_1$.

Let $\PP_1$ be the set of all $P\in\PP$ such that $w_1\not\in V(P)$.  We may assume that $\PP_1$ is nonempty, or else the lemma is satisfied with $A=\{w_1\}$ and $\PP'$ empty.  Let $Q_2$ be a minimal subpath of $Q_1$ such that $\descendants{Q_2}{Q}{X}$ has the core capture property for $\PP_1$.  We claim there is a vertex $w_2\in V(G)$ such that $w_2\ne w_1$ and $Q_2\subseteq S(w_2)$.  Indeed, if $Q_2$ is a single vertex $y$, then since $|B(y)| \ge 2$ and we may choose $w_2\in B(y)$ distinct from $w_1$.  Otherwise, we apply \Cref{lem:spanlongpath} with $\HH=\PP_1$ to obtain $w_2$.  Since $w_2$ is chosen from the union of two paths in $\PP_1$, neither of which contains $w_1$, we have $w_2\ne w_1$.  Let $\PP_2$ be the set of paths in $\PP_1$ that do not contain $w_2$.  Since $S(w_2)$ contains $Q_2$ and all paths in $\PP_2$ avoid $w_2$, it follows that the endpoints of each path in $\PP_2$ have subtrees that are disjoint from $V(Q_2)$ (or else appending $w_2$ would extend the path).

Let $\PP_3$ be the set of paths $P\in \PP_2$ that have an endpoint $v$ such that $S(v)\subseteq \descendants{Q_2}{Q}{X} - V(Q_2)$ and let $\PP_4 = \PP_2 - \PP_3$.  Our goal is to apply \Cref{cor:roundpath} to obtain a small set of vertices $B$ such that every path in $\PP_4$ intersects $B$.  If $\PP_4=\emptyset$, then we may simply take $B=\emptyset$.  Suppose $\PP_4$ is nonempty.  Let $\RR$ be the family of paths $R$ in $G - \{w_1,w_2\}$ of size at least $3$ such that each endpoint of $R$ has a subtree intersecting $Q_2$ and each interior vertex $u$ of $R$ satisfies $S(u)\subseteq \descendants{Q_2}{Q}{X} - V(Q_2)$.  We claim that if $P\in\PP_4$, then $P$ has a subpath in $\RR$.  Since $P$ has a core vertex in $V(\descendants{Q_2}{Q}{X})$ but no core vertex in $V(Q_2)$, it follows from \Cref{lem:boundary-fence} that $P$ contains a vertex $u$ such that $S(u)\subseteq \descendants{Q_2}{Q}{X} - V(Q_2)$.  Since $P\not\in \PP_3$, it follows that $u$ is an interior vertex in a subpath of $P$ contained in $\RR$.  Let $R$ be a longest path in $\RR$.  Applying \Cref{lem:general-two-way} to $\PP_4$ with glue vertices $W=\{w_1,w_2\}$, it follows that each path in $\PP_4$ intersects $R$.  Let $Y'$ be the component of $\descendants{Q_2}{Q}{X} - V(Q_2)$ such that each interior vertex $u$ of $R$ has a subtree $S(u)$ contained in $Y'$. Let $w_3$ and $w_4$ be the endpoints of $R$, and observe that each $P\in \PP_4$ that avoids $w_3$ and $w_4$ must contain an interior vertex in $R$ and hence have a core vertex in $Y'$.  Let $y'$ be the root of $Y'$.  Note that if $u$ is an endpoint of a path $P\in \PP_4$, then $S(u)$ contains a vertex outside $\descendants{Q_2}{Q}{X} - V(Q_2)$.  If $S(u)$ intersects $Q_2$, then we may append $w_2$ to $u$ to obtain a longer path.  Hence $S(u)$ contains a vertex outside $Y'$.  By \Cref{cor:roundpath} and \Cref{prop:ccg-mono}, there exists $B\subseteq V(G)$ such that $|B|\le 4(\ccg{Y',y'}) \le 4(\ccg{X,x})$ and each path in $\PP_4$ contains a vertex in $B\cup \{w_3,w_4\}$.

Our next goal is to process $\PP_3$.  If $\PP_3 = \emptyset$, then the lemma is satisfied with $A=\{w_1,w_2,w_3,w_4\}\cup B$ and $\PP'$ empty.  So we may assume that $\PP_3$ is nonempty.  If $\PP_3$ contains a path $P$ such that $S(u)\subseteq \descendants{Q_2}{Q}{X}-V(Q_2)$ for each vertex $u$ in $P$, then let $X'$ be the component $\descendants{Q_2}{Q}{X}-V(Q_2)$ containing the subtree of each vertex in $P$.  In this case, the lemma is satisfied with $A=\{w_1,w_2,w_3,w_4\}\cup B$ and $\PP' = \PP_3$.  Otherwise each path $P\in\PP_3$ has an endpoint $u$ with $S(u) \subseteq \descendants{Q_2}{Q}{X}-V(Q_2)$ and some other vertex $v$ with $S(v)$ intersecting $V(Q_2)$.

Let $\RR'$ be the set of paths $R$ in $G - w_2$ of size at least $2$ such that some endpoint $u$ has a subtree $S(u)$ that intersects $V(Q_2)$ but all other vertices $v$ in $R$ satisfy $S(v)\subseteq \descendants{Q_2}{Q}{X}-V(Q_2)$.  Note that each path in $\PP_3$ has a subpath in $\RR'$, and it follows from \Cref{lem:one-way} with $w=w_2$ that a longest path $R\in\RR'$ intersects each path in $\PP_3$.  Let $X'$ be the component of $\descendants{x}{Q_2}{X} - V(Q_2)$ such that all vertices in $R$ except one endpoint have a subtree contained in $X'$.  Let $w_5$ be the endpoint of $R$ such that $S(w_5)$ intersects $V(Q_2)$, and note that each path in $\PP_3$ that avoids $w_5$ contains a vertex from $R$ whose subtree is contained in $X'$.  Let $\PP'$ be the set of paths $P\in \PP_3$ that avoid $w_5$, and note that $X'$ has the core capture property for $\PP'$.

Let $A=\{w_1,w_2,w_3,w_4,w_5\}\cup B$ and let $P\in\PP$.  If $P$ avoids $w_1$ and $w_2$, then $P\in \PP_2 = \PP_3\cup \PP_4$.  If $P\in\PP_4$, then $P$ contains a vertex in $B$.  If $P\in \PP_3$, then either $P$ contains $w_5$ or $P\in\PP'$.
\end{proof}

Next, we give our upper bound on $\lpt(G)$ for a connected chordal graph $G$ in terms of the Cutter--Chooser game.

\begin{theorem}\label{thm:lpt-game}
If $G$ is a connected chordal graph with minimal tree representation $T$, then $\lpt(G) \le (4\ccg{T}+5)(\ccg{T})$.
\end{theorem}
\begin{proof}        
    Let $r$ be the vertex in $T$ that minimizes $\ccg{T,r}$; we view $T$ as a rooted tree with root $r$.  Let $\PP$ be the family of longest paths in $G$.  We may assume that each $P\in\PP$ has an edge, or else $G$ is a single vertex and $\lpt(G)=1$.  Since each $P\in\PP$ has an edge and therefore a nonempty core, it follows that $T$ has the core capture property for $\PP$.  
    
    We apply \Cref{lem:allpaths} once per round of a simulated Cutter--Chooser game to obtain the sequence $(\PP_0,X_0), \ldots, (\PP_t,X_t)$ and $A_1,\ldots,A_t$ such that $\PP = \PP_0 \supseteq \cdots \supseteq \PP_t$, $T=X_0 \supseteq \cdots \supseteq X_t$, the subtree $X_i$ has the core capture property for $\PP_i$ for $i \ge 0$, and each path in $\PP_{i-1} \setminus \PP_{i}$ intersects $A_i$ for $i\ge 1$.  In our simulated game, we use an optimal strategy for Cutter, and we have Chooser play according to the choice of $X'$ made by \Cref{lem:allpaths}.  Initially, we set $(\PP_0,X_0) = (\PP, X)$.  
    
    At the start of round $i$, Cutter is presented with the subtree $X_{i-1}$ rooted at $x_{i-1}$.  Let $z\in X_{i-1}$ be the vertex selected by Cutter in the simulated game, and let $Q$ be the $xz$-path in $X_{i-1}$. By \Cref{lem:allpaths}, there is a set $A_i\subseteq V(G)$ with $|A_i| \le 4\ccg{X_{i-1},x_{i-1}} + 5$ such that the family $\PP_i$ of paths in $\PP_{i-1}$ that are disjoint from $A_i$ is empty or there is a component $X_i$ of $X_{i-1} - V(Q)$ that has the core capture property for $\PP_i$.  Note that by \Cref{prop:ccg-mono}, we have $|A_i| \le 4\ccg{X_{i-1},x_{i-1}}+ 5 \le 4\ccg{T,r}+5 \le 4\ccg{T} + 5$.  If $\PP_i = \emptyset$, then the iteration ends with $t=i$ (and the choice of $X_t$ is arbitrary).  Otherwise, we update our game to have Chooser pick the component $X_i$ of $X_{i-1} - V(Q)$ and begin the next round.
    
    The iteration ends with a pair $(\PP_t,X_t)$ such that $\PP_t$ is empty.  Let $A=\bigcup_{i=1}^t A_i$ and note that $A$ is a longest path transversal for $G$.  Since each $A_i$ has size at most $4\ccg{T}+5$ and $t\le \ccg{T}$, the bound follows.  
\end{proof}

When $\ccg{T}$ is bounded by a constant, \Cref{thm:lpt-game} gives constant longest path transversals, as in the following corollary.

\begin{cor}\label{cor:lpt_subdiv_caterpillar}
If $G$ is a connected chordal graph admitting a host tree $T$ such that $T$ is a subdivided caterpillar, then $\lpt(G)\le 26$.
\end{cor}
\begin{proof}
    When $T$ is a subdivided caterpillar, we have $\ccg{T}\le 2$, and the bound follows.
\end{proof}

We are finally able to give our proof for \Cref{thm:lpt_bound}. We make no attempt to optimize the multiplicative constant $4$ on the leading $\lg^2 n$ term.  

\begin{theorem}\label{thm:lpt_bound}
    If $G$ is an connected $n$-vertex chordal graph, then $\lpt(G)\le 4\lg^2 n + O(\log n)$.  
\end{theorem}
\begin{proof}
    Let $T$ be a minimal tree representation for $G$.  Since $T$ is a minimal tree representation, the vertices in $T$ correspond to maximal cliques in $G$, and $G$ has at most $n$ maximal cliques (see  \cite{G72}).  Hence $|V(T)|\le n$.  It follows from \Cref{prop:ccg-bound} and \Cref{thm:lpt-game} that $\lpt(G)\le (4(1+\floor{\lg n}) + 5)(1+\floor{\lg n}) \le 4\lg^2 n + O(\log n)$.
\end{proof}

\section{Leafage Bounds} \label{sec:leafage}

A \emph{leaf} in a graph is a vertex of degree $1$.  The \emph{leafage} of a chordal graph $G$, denoted $\leafage{G}$, is the minimum number of leaves in a tree representation of $G$. Clearly, $\leafage{G} = 0$ if and only if $G$ is complete, and $\leafage{G}\le 2$ if and only if $G$ is an interval graph.  The notion of leafage was introduced by Lin, McKee, and West~\cite{LTW98} as a measure on how far a chordal graph is from being an interval graph. Habib and Stacho~\cite{HS09} gave a polynomial time algorithm for computing $\leafage{G}$.  Since each connected interval graph $G$ satisfies $\lpt(G)=1$, it is natural to expect that connected chordal graphs with small leafage have small longest path transversals.

Let $T$ be a tree with root $r$.  If $uv\in E(T)$ and $u$ is on the $rv$-path in $T$, then we say that $v$ is a \emph{child} of $u$.  A vertex with no children is \emph{childless}.  Note that for each vertex $u\in V(T)$ except $r$, we have that $u$ is a leaf in $T$ if and only if $u$ is childless.  It follows that the number of leaves and the number of childless vertices differs by at most $1$.

\begin{lemma}\label{lem:tree_path_divide_leaves}
    Let $T$ be a tree with root $r$ and $k$ childless vertices.  There exists a path $P$ such that $r$ is an endpoint of $P$ and every component of $T - V(P)$ has at most $k/2$ childless vertices.
\end{lemma}

\begin{proof}
    For each vertex $v\in V(T)$, let $c_v$ be the number of childless vertices in $T$ that belong to the subtree rooted at $v$.  The set of vertices $v$ such that $c_v > k/2$ induces a path $P$ with $r$ as an endpoint, and each component of $T-V(P)$ has at most $k/2$ childless vertices.
\end{proof}

As Cutter can use Lemma~\ref{lem:tree_path_divide_leaves} to select a path in the Cutter-Chooser game, the following is an immediate corollary to Theorem~\ref{thm:lct-upper-game} and Theorem~\ref{thm:lpt-game}. We omit the details.

\begin{cor}\label{cor:ccg_leafage_bound}
    Let $G$ be a chordal graph and let $m = \ell(G)$.  If $G$ is connected, then $\lpt(G) \le O(\log^2 m)$.  If $G$ is $2$-connected, then $\lct(G) \le O(\log m)$.
\end{cor}

We now proceed with a slightly more delicate analysis regarding leafage. For a tree $T$ on at least $2$ vertices, let $\mmf{T}$ be the maximum, over all $e\in E(T)$, of the minimum number of leaves of $T$ contained in a component of $T-e$.  For a connected chordal graph $G$ with tree representation $T$, we show that $\lpt(G)\le \mmf{T}$. Balister, Gy\"ori, Lehel, and Schelp \cite{BGLS04} observed that tree representations of connected chordal graphs contain longest path transversal bags; we include a proof for completeness. 
\begin{proposition}[Balister--Gy\"ori--Lehel--Schelp \cite{BGLS04}]\label{prop:LPTBag}
    Let $T$ be a tree representation of a connected chordal graph $G$.  There is a vertex $x\in V(T)$ such that $\bag{x}$ is a longest path transversal of $G$.
\end{proposition}
\begin{proof}
    Let $\PP$ be the family of longest paths in a connected chordal graph $G$. As longest paths are pairwise intersecting, for $P,Q\in \PP$, there exists $w \in V(P) \cap V(Q)$.  It follows that $\subtree{P}$ and $\subtree{Q}$ both contain $\subtree{w}$.  Therefore the set $\{\subtree{P}\st P\in\PP\}$ is a pairwise intersecting family of subtrees of $T$.  Since subtrees of a tree have the Helly property \cite{G04}, there exists $x\in V(T)$ such that for each $P\in\PP$, the vertex $x$ belongs to $\subtree{P}$.  It follows that $\bag{x}$ is a longest path transversal for $G$. 
\end{proof}

Let $T$ be a tree with $k$ leaves.  For $x,y\in V(T)$, we let $T[x,y]$ denote the path in $T$ with endpoints $x$ and $y$.  A \emph{finger} of $x$ is a path of the form $T[x,z]$, where $z$ is a leaf in $T$.  Note that every vertex of $T$ has $k$ fingers. Let $G$ be a connected chordal graph, $T$ be a minimal tree representation of $G$, and $x \in V(T)$.  A path $u_1\ldots u_t$ of $G$ is \emph{handy} with respect to $x$ if $u_1 \in \bag{x}$ but $u_j \not\in \bag{x}$ for $j>1$.  Note that if $Q$ is handy with respect to $x$ and $Q$ has initial vertex $u$, then $\subtree{Q-u}$ is contained in a component of $T-x$.  

\begin{proposition}\label{prop:Replace}
    Let $T$ be a tree representation for a connected chordal graph $G$, and let $x \in V(T)$ such that $\bag{x}$ is a longest path transversal of $G$.  If $Q$ is a handy path with respect to $x$ of maximum size, then $V(Q)$ is a longest path transversal in $G$.
\end{proposition}
\begin{proof}
Let $Q$ be a handy path with respect to $x$ of maximum size, with initial vertex $u$.  Suppose for a contradiction that $P$ is a longest path in $G$ that is disjoint from $Q$.  Since $\bag{x}$ is a longest path transversal, it follows that $P$ contains a vertex in $\bag{x}$.  Let $P_0$ be the suffix of $P$ whose initial vertex $w$ is in $\bag{x}$ but no other vertex in $P_0$ is in $\bag{x}$.  Since $P_0$ is handy with respect to $x$, we have that $|V(P_0)| \le |V(Q)|$ and so $Q$ has more vertices than $P_0 - w$.  Since $w,u \in \bag{x}$, we have that $wu\in E(G)$.  Replacing the suffix $P_0 - w$ of $P$ with $Q$ gives a longer path in $G$, contradicting the maximality of $P$.
\end{proof}

Let $T$ be a tree representation for a chordal graph $G$ and let $x,y\in V(T)$.  We say that $u\in V(G)$ is \emph{maximal from $x$ toward $y$} if $u \in \bag{x}$ and among all vertices $w\in \bag{x}$, the vertex $u$ maximizes $|V(S(w)) \cap V(T[x,y])|$. 

\begin{lemma}\label{lem:maxtri}
Let $G$ be a connected chordal graph with tree representation $T$, and let $ww'\in E(G)$.  Let $x\in V(T)$ and suppose that $w\in \bag{x}$ but $w'\not\in\bag{x}$.  Let $X$ be the component of $T-x$ containing $\subtree{w'}$.  There is a leaf $y$ in $T$ belonging to $X$ such that every vertex in $G$ distinct from $w$ which is maximal from $x$ toward $y$ completes a triangle with $ww'$.
\end{lemma}
\begin{proof}
    Since $ww'\in E(G)$, there is a vertex $z$ in $\subtree{w}\cap\subtree{w'}$, and note that $z\in V(X)$ also.  Let $y$ be a leaf in $T$ such that $z$ is on $T[x,y]$, and let $u$ be a vertex in $G$ such that $u\ne w$ but $u$ is maximal from $x$ toward $y$.  Since $z \in V(\subtree{w}) \cap V(T[x,y])$ and $u$ is maximal from $x$ to $y$, it follows that $z\in V(\subtree{u})$ also.  As $z$ is common to $\subtree{w}$, $\subtree{w'}$, and $\subtree{u}$, it follows that $\{w,w',u\}$ is a triangle in $G$.
\end{proof}

\begin{lemma}{\label{lem:Hand}} Let $G$ be a connected chordal graph with tree representation $T$ and let $x \in V(T)$ such that $\bag{x}$ is a longest path transversal of $G$. Let $Q$ be a handy path with respect to $x$ of maximum size with initial vertex $v$, and suppose that $|V(Q)| \ge 2$.  Let $X$ be the component of $T-x$ containing $\subtree{Q-v}$, and let $y_1,\ldots,y_k$ be the leaf vertices of $T$ in $X$.  For each $i$, let $u_i$ be a vertex in $\bag{x}$ that is maximal from $x$ toward $y_i$.  The set $\set{u_1,\ldots, u_k}$ is a longest path transversal for $G$.
\end{lemma}

\begin{proof}
   Let $A=\set{u_1,\ldots,u_k}$, and let $Q=v_1\ldots v_t$, where $v=v_1 \in \bag{x}$.  We show that we may assume that $v_1\in A$.  Otherwise, since $v_1v_2 \in E(G)$ with $v_1\in \bag{x}$ and $v_2\not\in\bag{x}$, it follows from \Cref{lem:maxtri} that some $u\in A$ (distinct from $v_1$ since $v_1\not\in A$) completes a triangle with $v_1v_2$.  Therefore $uv_2\ldots u_k$ is also a handy path with respect to $x$ of maximum size whose initial vertex is in $A$.  
   
  Let $P$ be a longest path in $G$ and suppose for a contradiction that $V(P) \cap A = \emptyset$.  By \Cref{prop:Replace}, we have $V(P) \cap V(Q) \ne \emptyset$, and so $P$ contains a vertex $v_i$ in $Q$.  Since $v_1\in A$ and $P$ is disjoint from $A$, it must be that $i\ge 2$ and so $\subtree{v_i}\subseteq X$.  Additionally, since $\bag{x}$ is a longest path transversal, it follows that $P$ contains adjacent vertices $ww'$ such that $w\in\bag{x}$ and $\subtree{w
  '} \ss X$.  By \Cref{lem:maxtri}, some vertex in $A$ completes a triangle with $w$ and $w'$, and we obtain a longer path by inserting this vertex in $P$ between $w$ and $w'$.    
\end{proof}

Recall that $\mmf{T}$ is the maximum, over all $e\in E(T)$, of the minimum number of leaves of $T$ contained in a component of $T-e$. 

\begin{theorem}{\label{thm:LeafageBound}} If $G$ is a connected chordal graph with minimal tree representation $T$, then either $G$ is complete or $\lpt(G) \le \mmf{T}$.
\end{theorem}
\begin{proof}
     If $|V(T)| = 1$, then $G$ is complete.  So we may assume $|V(T)|\ge 2$.  
    
    We show that for each $x\in V(T)$, a handy path with respect to $x$ of maximum size has at least $2$ vertices.  Let $x\in V(T)$ and let $y$ be a neighbor of $x$.  By minimality of $T$, we have that $v_y\in \bag{y}-\bag{x}$ and $v_x\in \bag{x}-\bag{y}$ for some $v_y,v_x\in V(G)$.  Every $v_xv_y$-path in $G$ contains a vertex in $\bag{x}\cap\bag{y}$.  Let $u\in \bag{x}\cap\bag{y}$, and note that $uv_y$ is a handy path with respect to $x$. 
    
    We construct an auxiliary digraph $H$ on $V(T)$ as follows.  For each $x\in V(T)$ such that $\bag{x}$ is a longest path transversal, we select a maximum handy path $Q$ with respect to $x$ with initial vertex $v$, and we direct an edge in $H$ from $x$ to $y$, where $y$ is the neighbor of $x$ belonging to the component of $T-x$ that contains $\subtree{Q-v}$.

    Note that in $H$, the outdegree of $x$ is $1$ if $\bag{x}$ is a longest path transversal of $G$ and $0$ otherwise.  By \Cref{prop:LPTBag}, there exists an edge in $H$.  Suppose $xy\in E(H)$.  By \Cref{lem:Hand}, we obtain a longest path transversal contained in $\bag{x}\cap\bag{y}$.  Hence $xy\in E(H)$ implies that $\bag{y}$ is a longest path transversal in $G$, and therefore $y$ also has outdegree at least $1$.  It follows that $H$ contains a directed cycle.  Since the underlying graph of $H$ is acyclic, it follows that $xy\in E(H)$ and $yx\in E(H)$ for some vertices $x,y\in V(H)$.

    Let $e$ be the edge $xy$ in $T$, let $T_x$ be the component of $T-e$ containing $x$, and let $T_y$ be the component of $T-e$ containing $y$.  Since $xy\in E(H)$, it follows from \Cref{lem:Hand} that maximal vertices from $x$ toward leaves in $T_y$ form a longest path transversal.  Similarly, maximal vertices from $y$ toward leaves in $T_x$ also form a longest path transversal.  Therefore $\lpt(G)\le \min\{m_x,m_y\} \le \mmf{T}$, where $m_x$ and $m_y$ are the number of leaves in $T$ belonging to $T_x$ and $T_y$, respectively.
\end{proof}

\begin{cor}\label{cor:substar}
The family of connected chordal graphs admitting a tree representation $T$ such that $T$ is a subdivided star is Gallai.
\end{cor}
\begin{proof}
    Let $G$ be a connected chordal graph with tree representation $T$ such that $T$ is a subdivision of a star.  Note that $\mmf{T} = 1$.  It follows from \Cref{thm:LeafageBound} that $G$ is complete (implying $\lpt(G)=1$) or $\lpt(G)\le\mmf{T}=1$.
\end{proof}

\section{Conclusion}\label{sec:conclusion}

We close with some open problems on longest path transversals, roughly organized from more difficult questions to easier ones.  In each case, there is a natural analogue for longest cycle transversals in $2$-connected graphs that is also open.

\begin{problem}[Walther~\cite{W69}, c.f. Zamfirescu~\cite{zamfirescu1972}]
    Is there a constant $C$ such that every connected graph $G$ satisfies $\lpt(G)\le C$?
\end{problem}

\begin{problem}[Balister--Gy\H{o}ri--Lehel--Schelp~\cite{BGLS04}]\label{prob:chordal-gallai}
Is it true that the family of chordal graphs is Gallai, meaning that $\lpt(G)=1$ for each connected chordal graph $G$?    
\end{problem}

If a positive resolution to \Cref{prob:chordal-gallai} remains elusive, constant upper bounds for chordal graphs would be very interesting.

\begin{problem}
Is there a constant $C$ such that every connected chordal graph $G$ satisfies $\lpt(G)\le C$? 
\end{problem}

By restricting the allowed host trees, we obtain various subfamilies of chordal graphs.  When the host tree is restricted to be a path, we obtain the interval graphs.  \Cref{cor:substar} shows that when the host tree is restricted to be a subdivided star, we still obtain a Gallai family.  

\begin{conj}\label{conj:caterpillar}
If $G$ is a connected chordal graph admitting a tree representation $T$ such that $T$ is a subdivided caterpillar, then $\lpt(G)=1$.
\end{conj}

From \Cref{cor:lpt_subdiv_caterpillar}, we have that $\lpt(G)\le 26$ when the hypotheses of \Cref{conj:caterpillar} apply.  Perhaps the easiest interesting special case of \Cref{conj:caterpillar} is the case that the host tree $T$ is a subdivided double star, meaning that $T$ has at most two vertices with degree at least $3$.

\section*{Acknowledgements}

The third author was supported in part by NSF RTG DMS-1937241 and an AMS-Simons Travel Grant.

\bibliographystyle{plain}
\bibliography{citations}
\end{document}